\documentclass{amsart}

\usepackage[normalem]{ulem}

\usepackage{lineno}

\usepackage{soul}

\usepackage{amssymb, amsmath, amsthm, euscript, color}
\usepackage[dvipsnames]{xcolor}

\usepackage{amscd}

\usepackage{bbm}

\usepackage{enumitem}

\usepackage[english]{babel}

\usepackage{hyperref}

\numberwithin{equation}{section}

\theoremstyle{plain}

\newtheorem{theorem}{Theorem}

\newtheorem{proposition}{Proposition}

\newtheorem{remark}{Remark}

\newtheorem{lemma}{Lemma}

\newtheorem{thm}{Theorem}

\newtheorem{pro}[thm]{Proposition}

\theoremstyle{definition}

\newtheorem{definition}{Definition}

\newtheorem{example}{Example}

\newcommand{\Q}{\mathbb{Q}}

\newcommand{\R}{\mathbb{R}}

\newcommand{\Z}{\mathbb{Z}}

\newcommand{\N}{\mathbb{N}}

\newcommand{\C}{\mathbb{C}}

\newcommand{\id}{\textnormal{id}}

\DeclareMathOperator{\Scal}{scal}

\DeclareMathOperator{\Def}{def}

\DeclareMathOperator{\Min}{min}

\DeclareMathOperator{\Spin}{Spin}

\begin{document}

\author{Agnese Mantione and Rafael Torres}

\title[Positive scalar curvature and homotopy types of 4-manifolds]{Geography of 4-manifolds with positive scalar curvature}

%\address{Universit\'a degli Studi di Trieste\\ Dipartimento di Matematica e Geoscienze\\Via Alfonso Valerio 12/1\\ 34127\\ Trieste\\Italy}

\address{Mathematisches Institut der Westf\"alischen Wilhelms-Universit\"at M\"unchen\\Einsteinstr. 62\\DE-48149\\M\"unster\\Germany}

\email{amantione@uni-muenster.de}

\address{Scuola Internazionale Superiori di Studi Avanzati (SISSA)\\ Via Bonomea 265\\ 34136\\ Trieste\\Italy}

\email{rtorres@sissa.it}

\subjclass[2020]{57M50, 53C21, 57R19}

\maketitle

\emph{Abstract}: We discuss the geography problem of closed oriented 4-manifolds that admit a Riemannian metric of positive scalar curvature, and use it to survey  mathematical work employed to address Gromov's observation that  manifolds with positive scalar curvature tend to be inessential by focusing on the four-dimensional case. We also point out an strengthening of a result of Carr and its extension to the non-orientable realm.

\section{Introduction}

The geography problem of 4-manifolds with positive scalar curvature for a group $G$ consists of determining the sets\begin{equation}\label{Geography PSC}\mathcal{G}^{\dagger}_{> 0}(G) = \{(\chi(M), \sigma(M)): M\in \mathcal{M}^{\dagger}_{> 0}(G)\},\end{equation}where $\chi(M)$ is the Euler characteristic, $\sigma(M)$ is the signature, and $\mathcal{M}^{\dagger}_{>0}(G)$ is the set of closed smooth oriented 4-manifolds $M$ whose fundamental group is isomorphic to $G$, with $\dagger$ specifying its $w_2$-type, and that admit a Riemannian metric $(M, g)$ of strictly positive scalar curvature $\Scal_g(p) > 0$ for every point $p\in M$. An orientable 4-manifold $M$ has $w_2$-type (I) and $\dagger = 1$ if $w_2(\widetilde{M})\neq 0$, $w_2$-type (II) and $\dagger = 2$  if $w_2(M) = 0$, and $w_2$-type (III) and $\dagger = 3$ if $w_2(M)\neq 0$ and $w_2(\widetilde{M}) = 0$, where $\widetilde{M}$ is the universal cover of $M$ \cite{[HambletonKreck]}. There is a division of the set of closed smooth oriented 4-manifolds\begin{equation}\mathcal{M}_{>0}(G) = \mathcal{M}^{1}_{>0}(G) \cup \mathcal{M}^{2}_{>0}(G)\cup \mathcal{M}^{3}_{>0}(G),\end{equation} and the aforementioned geography problem becomes the computation of the set\begin{equation}\mathcal{G}_{> 0}(G) = \mathcal{G}^{1}_{> 0}(G) \cup \mathcal{G}^{2}_{> 0}(G) \cup \mathcal{G}^{3}_{> 0}(G).\end{equation}
Given an element $\alpha\in H_4(G) = H_4(BG; \Q)$, let $\mathcal{M}^{\dagger}_{> 0}(G, \alpha)$ be the set of pairs $(M, f)$ for $M\in \mathcal{M}^{\dagger}_{> 0}(G)$ with reference map $f: M\rightarrow BG$ into the corresponding Eilenberg-MacLane space inducing an isomorphism $f_{\ast}:\pi_1(M)\rightarrow G$  and such that $f_{\ast}([M]) = \alpha$ for $\dagger = 1, 2, 3$. Define\begin{equation}\label{Geography Alpha 1}\mathcal{G}^{\dagger}_{> 0}(G, \alpha):= \{(\chi(M), \sigma(M)) : (M, f)\in \mathcal{M}^{\dagger}_{> 0}(G, \alpha)\};\end{equation}see \cite[Section 2]{[KirkLivingston2]}. The set (\ref{Geography PSC}) consists of the union of the sets (\ref{Geography Alpha 1}) for $\alpha\in H_4(G)$.

We consider some refinements adapted to our purposes of the classic Hausmann-Weinberger invariant \cite[\S 1]{[HausmannWeinberger]}\begin{equation}\label{HW Invariant}q_{> 0}^{\dagger}(G):= \Min \{\chi(M): M\in \mathcal{M}^{\dagger}_{> 0}(G)\}\end{equation}and\begin{equation}\label{HW Invariant}q_{> 0}(G):= \Min \{q_{> 0}^{\dagger}(G) : \dagger = 1,2,3\}\end{equation}as done by Kirk-Livingston \cite[Definition 3.3]{[KirkLivingston2]}. We define\begin{equation}\label{Function1}q_{> 0, \alpha}^{\dagger}(G, \sigma):= \Min \{\chi(M): (M, f)\in \mathcal{M}^{\dagger}_{> 0}(G, \alpha) \; \text{and} \space \; \sigma(M) = \sigma\}\end{equation} and\begin{equation} q_{> 0, \alpha}(G, \sigma):= \Min \{q_{> 0, \alpha}^{\dagger}(G, \sigma): \dagger = 1, 2, 3\}\end{equation} by taking into account the signature and the reference map as done by Kotschick \cite{[Kotschick]} and Kirk-Livingston \cite[Definition 3.3]{[KirkLivingston2]}.

%\begin{equation}\label{SHW Invariant}q_{> 0}^{\dagger}(G, \sigma):= \Min \{\chi(M): M\in \mathcal{M}_{> 0}^{\dagger}(G) \; \text{and} \space \; \sigma(M) = \sigma\}\end{equation}and\begin{equation}q_{> 0}(G, \sigma):= \Min \{q_{> 0}^{\dagger}(G, \sigma) : \dagger = 0, 1\} \end{equation} by taking the signature into account as Kotschick does in \cite{[Kotschick]}, and define the related functions\begin{equation}\label{Function1}q_{> 0, \alpha}^{\dagger}(G, \sigma):= \Min \{\chi(M): (M, f)\in \mathcal{M}^{\dagger}_{> 0}(G, \alpha) \; \text{and} \space \; \sigma(M) = \sigma\}\end{equation} and\begin{equation} q_{> 0, \alpha}(G, \sigma):= \Min \{q_{> 0, \alpha}^{\dagger}(G, \sigma): \dagger = 0, 1\}\end{equation} by taking into account the reference map; cf. \cite[Definition 3.3]{[KirkLivingston2]}. 

In particular,\begin{equation}\label{Relation}q_{> 0}^{\dagger}(G) = \Min \{q_{> 0, \alpha}^{\dagger}(G, \sigma) : \sigma \in \Z, \alpha\in H_4(G)\},\end{equation}and\begin{equation}\label{Relation2}q_{> 0}(G) = \Min \{q_{> 0, \alpha}^{\dagger}(G, \sigma) : \sigma \in \Z, \alpha\in H_4(G), \dagger = 1,2,3\},\end{equation}and the invariant $q_{> 0, \alpha}^{\dagger}(G, \sigma)$ determines $\mathcal{G}_{> 0}^{\dagger}(G, \alpha)$ and viceversa for a given $w_2$-type $\dagger = 1, 2, 3$ \cite[Theorem 3.4 Item (5)]{[KirkLivingston2]}. 

A closed orientable $n$-manifold $M$ with fundamental group $G = \pi_1(M)$ is said to be an essential manifold if $f_\ast([M])\neq 0\in H_n(G)$. An observation of Gromov is that manifolds that support a Riemannian metric of positive scalar curvature tend to be inessential. In this note, we address Gromov's observation for 4-manifolds in terms of the geography problem described before. We do so without involving Seiberg-Witten theory \cite[\S 2.4]{[GompfStipsicz]}. The contribution (albeit modest) of this survey is to describe how work of Carr \cite{[Carr]}, Cecchini-Schick \cite{[CecchiniSchick]}, Chodosh-Li \cite{[ChodoshLi]}, Gromov \cite{[Gromov]}, Gromov-Lawson \cite{[GromovLawson1], [GromovLawson2], [GromovLawson3]}, Kirk-Livingston \cite{[KirkLivingston2]}, Rosenberg \cite{[Rosenberg1],[Rosenberg2]}, Schick \cite{[Schick]} and Schoen-Yau \cite{[SchoenYau]} turns the study of the geography problem of closed 4-manifolds with positive scalar curvature into a problem in group (co)homology and geometric group theory. We provide  a stronger version to Carr's result in Theorem \ref{Theorem Carr} and its extension to non-orientable manifolds in Theorem \ref{Theorem NonOrientable Carr}.

There are two canonical steps to study the geography problem, which are reviewed in Section \ref{Section NPositiveScalarCurvature} and Section \ref{Section Constructions}. The first step is the construction of manifolds: a fundamental result of Gromov-Lawson and Schoen-Yau (Theorem \ref{Theorem PSC}) is used to construct a myriad of examples of 4-manifolds whose Euler characteristic and signature realize the sets (\ref{Geography Alpha 1}) for $\alpha = 0\in H_4(G)$. In the same vein, we use a result due to Carr that shows that the fundamental group is not an obstruction to the existence of a metric of positive scalar curvature in Theorem \ref{Theorem Carr}. We state a stronger result than the one originally considered by Carr and compute some topological invariants of the manifolds. As a warm-up case of the first step, we now describe the geography and computation of the invariants that were introduced for the case of surface groups $G = \pi_1(\Sigma_g)$, where $\Sigma_g$ is a closed orientable surface of genus $g\in \N$.

\begin{example}\label{Geography Surface Groups} Geography problem of 4-manifolds with positive scalar curvature and fundamental group isomorphic to a non-trivial surface group. The Eilenberg-MacLane space for these groups is $B\pi_1(\Sigma_g) = \Sigma_g$ and\begin{equation}H_4(\pi_1(\Sigma_g); \Q) = H_4(B\pi_1(\Sigma_g); \Q) = H_4(\Sigma_g; \Q) = 0.\end{equation}Hence, $\alpha = 0$ is the only case that needs to be considered. The total space $S^2\widetilde{\times} \Sigma_g$ of the non-trivial 2-sphere bundle over $\Sigma_g$ and the product $S^2\times \Sigma_g$ admit Riemannian metrics of positive scalar curvature \cite[1.4 Observation]{[Stolz]}. The values of the modified Hausmann-Weinberger invariants are\begin{equation}\chi(S^2\times \Sigma_g) = q_{> 0}^{2}(\pi_1(\Sigma_g)) = 4 - 4g = q_{> 0}^{3}(\pi_1(\Sigma_g)) = \chi(S^2\widetilde{\times} \Sigma_g)\end{equation} and $q_{> 0}(\pi_1(\Sigma_g)) = 4 - 4g$ by a result of Kotschick \cite{[Kotschick]}; see \cite[Section 3]{[KirkLivingston2]}. An index theoretical obstruction due to Lichnerowicz and Schr\"odinger states that the signature of a closed orientable Riemannian 4-manifold $(M, g)$ with scalar curvature $\Scal_g > 0$ and second Stiefel-Whitney class $w_2(M) = 0$ must satisfy $\sigma(M) = 0$; see Theorem \ref{Theorem Lichnerowicz} and Section \ref{Section NPositiveScalarCurvature}. Therefore, the identitities\begin{equation}q_{> 0}^{2}(\pi_1(\Sigma_g), \sigma) = q_{> 0}^{2}(\pi_1(\Sigma_g), 0)  = q_{> 0}^{2}(\pi_1(\Sigma_g)) = q_{> 0}(\pi_1(\Sigma_g)) \end{equation} hold.
By taking a connected sum $M_{2k + 2}:= S^2\times \Sigma_g\#k(S^2\times S^2)$ with $k\in \N$ copies of the product of two round 2-spheres we obtain a closed simply connected 4-manifold with a Riemannian metric of positive scalar curvature, $w_2(M_{2k + 2}) = 0$, Euler characteristic $\chi(M_{2k + 2}) = 4 - 4g + 2k$, and signature $\sigma(M_{2k + 2}) = 0$. A fundamental result of Gromov-Lawson and Schoen-Yau states that the existence of a Riemannian metric with positive scalar curvature is a property that is closed under connected sums; see Theorem \ref{Theorem PSC}. We conclude that\begin{equation}\mathcal{G}^{2}_{> 0}(\pi_1(\Sigma_g)) = \{(4 - 4g + 2k, 0): k\in \Z_{\geq 0}\} =  \mathcal{G}^{2}_{> 0}(\pi_1(\Sigma_g), 0).\end{equation}

The scenario in the case of non-zero second Stiefel-Whitney class is addressed analogously by taking connected sums with the complex projective plane $\mathbb{CP}^2$ equipped with the Fubini-Study metric and the underlying 4-manifold taken with the opposite orientation $\overline{\mathbb{CP}^2}$. Gromov-Lawson's and Schoen-Yau's result implies that the connected sum $M_{a, b}: = S^2\widetilde{\times} \Sigma_g\# a(\mathbb{CP}^2)\# b(\overline{\mathbb{CP}^2})$ admits a Riemannian metric of positive scalar curvature. We have that\begin{equation}q_{> 0}(\pi_1(\Sigma_g), \sigma)  =  4 - 4g + |\sigma|\end{equation}and\begin{equation} \mathcal{G}_{> 0}(\pi_1(\Sigma_g)) =  \{(4 - 4g + |k|, k): k\in \Z\}.\end{equation}
\end{example}

The second step in the study of the geography problem is based on three fundamental obstruction techniques, and its goal is to show that the sets (\ref{Geography Alpha 1}) are empty for $\alpha \neq 0\in H_4(G)$. We now briefly mention these techniques focusing on dimension four, and we recall them in further detail in Section \ref{Section NPositiveScalarCurvature}. We start with an index theoretical technique that has arguably attracted the most attention of researchers through the years. Rosenberg \cite{[Rosenberg1]}, \cite[Section 1.1.1]{[Rosenberg2]} employed the $C^\ast$-algebraic index of the Dirac operator on a closed orientable manifold with zero Stiefel-Whitney class.  For an essential 4-manifold $M$ with second Stiefel-Whitney class $w_2(M) = 0$, his result imposes restrictions on the sets (\ref{Geography Alpha 1}) and (\ref{Geography PSC}); see Theorem \ref{Theorem Rosenberg}. These results yield the following theorem.

\begin{thm}\label{Theorem A} Let $G$ be a finitely presented group whose first and second Betti numbers satisfy\begin{equation}\label{Hypothesis 1}b_1(G) - b_2(G) = \Def(G),\end{equation}where $\Def(G)$ is its deficiency. Suppose that $G$ satisfies the strong Novikov conjecture.

The sets $\mathcal{G}^{2}_{> 0}(G, \alpha)$ are empty for every $\alpha \neq 0$, we have that\begin{equation}\label{Value Minimizer}\mathcal{G}_{> 0}^{2}(G)  = \{(2 - 2\Def(G) + 2|k|, 0): k\in \Z\} =  \mathcal{G}^{2}_{> 0}(G, 0),\end{equation} and the values of the modified Hausmann-Weinberger invariants are\begin{equation}\label{Value Minimizer1} q_{> 0}^2(G) = 2 - 2\Def(G) = q_{> 0, 0}^2(G, 0).\end{equation}
\end{thm}

Rosenberg's result and the strong Novikov conjecture  are briefly recalled in Section \ref{Section NPositiveScalarCurvature}. Gromov-Lawson introduced a method to make use of the index of the Dirac operator to obstruct positive scalar curvature through the notion of an enlargeable manifold in \cite{[GromovLawson1]}. While their techniques where originally deployed for orientable manifolds with $w_2 = 0$, they dropped this hypothesis condition in \cite{[GromovLawson3]}; cf. Cecchini-Schick \cite{[CecchiniSchick]}.

Schoen-Yau's stable minimal surface technique \cite{[SchoenYau]} does not make any assumption on the second Stiefel-Whitney class, and it has proven to be a powerful obstruction technique provided that the manifold and its fundamental group satisfy the following property.

\begin{definition}\label{Definition Property NPSC}A group $G$ satisfies Property NPSC if $H^1(G; \Z) = \Z^s$ for $s\geq 2$ with generators $\{\theta_1, \theta_2\}$ that satisfy\begin{equation}\label{Element}\theta_1\cap \theta_2\cap f_\ast([M])\neq 0\in H_2(BG; \Z)\end{equation}for $f:M\rightarrow BG$ the reference map of a closed orientable 4-manifold $M$ with $\pi_1(M) = G$. %In particular, the element (\ref{Element}) does not lie in the image of the Hurewicz map $\pi_2(BG)\rightarrow H_2(BG; \Z)$. 
\end{definition}

Section \ref{Section NPositiveScalarCurvature} contains further details on minimal stable hypersurfaces and the formal obstruction to the existence of a metric of positive scalar curvature that arises from them as observed by Schick \cite{[Schick]}; see Theorem \ref{Theorem Schick}. We show in Example \ref{Example RAAGs} that right-angled Artin groups satisfy Property NPSC. 

A central conjecture in the study of geometry and topology of Riemannian manifolds of positive scalar curvature states that a closed aspherical $n$-manifold does not admit such a metric. The case $n = 2$ corresponds to the Gauss-Bonnet theorem and the case $n = 3$ has been proven by Gromov-Lawson \cite{[GromovLawson1], [GromovLawson2], [GromovLawson3]} and Schoen-Yau \cite{[SchoenYau], [SchoenYau1]} along with grounbreaking results of Perelman \cite{[Perelman1], [Perelman2], [Perelman3]}.  Schoen-Yau proposed a strategy for the case $n = 4$ in \cite[Theorem 6]{[SchoenYau2]} and Chodosh-Li followed Schoen-Yau's strategy showed that a closed aspherical 4-manifold does not admit a Riemannian metric of positive scalar curvature \cite[Theorem 2]{[ChodoshLi]} using generalized soap bubbles: submanifolds that are stationary for prescribed mean-curvature functionals. We point out in Lemma \ref{Lemma Gromov} that recent work of Gromov \cite[Main Theorem]{[Gromov]} and Chodosh-Li \cite[Theorem 2]{[ChodoshLi]} regarding non-existence of Riemannian metrics of positive scalar curvature on closed aspherical 5-manifolds implies the corresponding statement on certain essential 4-manifolds.

An application of the stable minimal hypersurfaces, enlargeability and generalized soap bubbles methods yield the following result.

\begin{thm}\label{Theorem B} Let $G$ be a finitely presented group whose first and second Betti numbers satisfy\begin{equation}\label{Hypothesis 1}b_1(G) - b_2(G) = \Def(G),\end{equation}where $\Def(G)$ is its deficiency. Suppose that either\begin{itemize}\item $H_4(G) = 0$,
\item $G$ satisfies Property NPSC of Definition \ref{Definition Property NPSC} or \item $G$ is the fundamental group of a closed orientable aspherical 4-manifold.
\end{itemize}

The sets (\ref{Geography Alpha 1}) are empty for every $\alpha \neq 0$. We have that\begin{equation}\label{Value Minimizer}\mathcal{G}_{> 0}(G)  = \{(2 - 2\Def(G) + |k|, k) : k\in \Z\}\end{equation}and\begin{equation}\label{Value Minimizer Spin}\mathcal{G}_{> 0}^2(G)  = \{(2 - 2\Def(G) + 2|k|, 0) : k\in \Z\}\end{equation}The values of the modified Hausmann-Weinberger invariants are\begin{equation}\label{Value Minimizer1} q_{> 0}(G) = 2 - 2\Def(G) = q_{> 0, 0}(G, 0)\end{equation}and\begin{equation}\label{Value Minimizer2}q_{> 0, 0}(G, \sigma) = |\sigma| + 2 - 2\Def(G) = q_{> 0}(G, \sigma).\end{equation}
\end{thm}

A list of groups that satisfy the hypothesis of Theorem \ref{Theorem B} includes right-angled Artin groups, Thompson's group, knot and 3-manifold groups, products of surface groups, several 2-knot and super perfect groups, solvable Baumslag-Solitar groups, and fundamental groups of Riemannian 4-manifolds whose sectional curvature is nonpositive. The last result that we mention samples how the geography problem of 4-manifolds with positive scalar curvature can be solved for several groups that do not satisfy Hypothesis (\ref{Hypothesis 1}).

\begin{pro}\label{Proposition C}Suppose $\gcd(p, q) \neq 1$ and let\begin{equation}\label{Examples Groups}G_{p, q} = \langle x, y : x^p = y^q = [x, y] = 1\rangle = \Z/p\oplus \Z/q.\end{equation}The identities\begin{equation}q_{> 0}(G_{p, q})  = 2 = q_{> 0}^2(G_{p, q}) \end{equation}and\begin{equation}q_{> 0, 0}(G_{p, q}, \sigma) = 2 + |\sigma|\end{equation}hold.
Moreover,\begin{equation}\label{Value Minimizer Proposition}\mathcal{G}_{> 0}(G_{p, q})  = \{(2 + |k|, k): k\in \Z_{\geq 0}\}\end{equation}and\begin{equation}
\mathcal{G}^{2}_{> 0}(G_{p, q}) = \{(2k + 2, 0): k\in \Z_{\geq 0}\} = \mathcal{G}^{2}_{> 0, 0}(G_{p, q}, 0).
\end{equation}
\end{pro}

The requirement $\gcd(p, q)\neq 1$ guarantees that the group (\ref{Examples Groups}) is not cyclic. None of the groups (\ref{Examples Groups}) satisfy Hypothesis (\ref{Hypothesis 1}) since their deficiency is -1.

We finish this introduction by mentioning that the geography problem of closed orientable 4-manifolds for a group $G$ and the Hausmann-Weinberger invariant themselves have not been determined except for a handful of groups; see Hausmann-Weinberger \cite{[HausmannWeinberger]}, Hildum \cite{[Hildum]}, Kirk-Livingston \cite{[KirkLivingston1], [KirkLivingston2]}, Kotschick \cite{[Kotschick]}.

\subsection{Acknowledgements:}We thank the referee for her/his detailed referee report, which helped us improve the manuscript. R. T. thanks Thomas Schick for useful e-mail correspondence. We are happy to join the  long list of mathematicians that express their gratitude to Misha Gromov for the wonderful mathematics throughout the years.

\section{Obstructions to the existence of a Riemannian metric of positive scalar curvature}\label{Section NPositiveScalarCurvature}

The Dirac operator $D$ on sections of the spinor bundle $S\rightarrow M$ of an orientable Riemannian 4-manifold $(M, g)$ with second Stiefel-Whitney class $w_2(M) = 0$ yields the Bochner-Lichnerowicz-Weitzenb\"ock formula\begin{equation}\label{BLW Formula}D^2 = \nabla^\ast \nabla + \frac{1}{4}\Scal_g\end{equation} where $\nabla$ is the connection on $S$, $\nabla^\ast$ its adjoint, and $\nabla^\ast \nabla$ is the connection Laplacian \cite[\* Chapter II]{[LawsonMichelsohn]}. The formula (\ref{BLW Formula}), due to work of Lichnerowicz \cite{[Lichnerowicz]} and Schr\"odinger \cite{[Schroedinger]}, indicates that the scalar curvature $\Scal_g$ is both an obstruction for $D$ to be the square root operator of the Laplacian, and an error term that arises when the identity $D^2 = \nabla^\ast \nabla$ is extended from the flat to the curved case.

The formula (\ref{BLW Formula}) states that a necessary condition for $\Scal_g > 0$ is that the kernel of the Dirac operator must be trivial and its index must vanish. The Atiyah-Singer Index Theorem \cite[Theorem 13.10, Chapter III]{[LawsonMichelsohn]} allows us to phrase this obstruction in terms of a rational Pontrjiagyn number of $M$ known as the $\hat{A}$-genus \cite[\S 8 Chapter II, \S 11 Chapter III]{[LawsonMichelsohn]}. 

\begin{theorem}\label{Theorem Lichnerowicz} Lichnerowicz \cite{[Lichnerowicz]}, Schr\"odinger \cite{[Schroedinger]}. Let $(M, g)$ be a closed orientable Riemannian 4-manifold with second Stiefel-Whitney class $w_2(M) = 0$. If $\Scal_g > 0$, then $\hat{A}(M) = 0$.
\end{theorem}

Hirzebruch's signature theorem \cite[Theorem 1.4.12]{[GompfStipsicz]} relates the $\hat{A}$-genus with the signature of the 4-manifold $M$ by $\hat{A}(M) = - \frac{1}{8}\sigma(M)$. In particular, the signature of a closed orientable Riemannian 4-manifold $(M, g)$ with $\Scal_g > 0$ is zero. 

We now briefly discuss a refinement of the $\hat{A}$-genus studied by Hitchin \cite{[Hitchin]} and Rosenberg \cite{[Rosenberg1]} that take into account the fundamental group of the manifold through the corresponding Eilenberg-MacLane space. Let $[M, f]\in \Omega^{\Spin}_n(BG)$ be the $\Spin$-bordism class of the pair $(M, f)$ where $f:M\rightarrow BG$, and and let $C^\ast_r(G)$ be the reduced $C^\ast$-algebra of the discrete group $G$. An invariant $\alpha(M)\in KO_n(C^\ast_r(G))$ is defined as the image of the fundamental class $[M, f]\in \Omega^{\Spin}(BG)$ under the map of generalized homology theories\begin{equation}\label{Sequence}\Omega^{\Spin}_\ast(BG)\overset{D}\longrightarrow ko_\ast(BG)\overset{per}\longrightarrow KO_\ast(BG)\overset{A}\longrightarrow KO_n(C^\ast_r(G)),\end{equation}where $ko_\ast(BG)$ and $KO_\ast(BG)$ are the periodic and connective real K-homology of the Eilenberg-MacLane space $BG$, respectively. That is,\begin{equation}\label{Obstruction Rosenberg}\alpha(M):= \alpha ([M, f]) = A\circ per\circ D([M, f]) = 0\in KO_n(C^\ast_r(G)).\end{equation}The invariant (\ref{Obstruction Rosenberg}) only depends on the bordism class $[M, f]$. For further details, we direct the reader towards the well-known surveys of Hanke \cite{[Hanke]}, Rosenberg \cite[Section 1.1.1]{[Rosenberg2]}, the chapter of Rosenberg-Stolz \cite{[RosenbergStolz]}, the lectures of Stolz \cite{[Stolz]}, and the book Lawson-Michelsohn \cite[Chapters II, III and IV]{[LawsonMichelsohn]}.

%\begin{conj}\label{Gromov-Lawson Conjecture}Gromov-Lawson-Rosenberg Conjecture \cite{[GromovLawson3]}. Let $(M, g)$ be a closed orientable $n$-manifold with $w_2(M) = 0$ and $n\geq 5$. There is a Riemannian metric $(M, g)$ whose scalar curvature satisfies $\Scal_g > 0$ if and only if\begin{equation}\label{Obstruction Rosenberg}\alpha(M) = \alpha ([M, f]) = A\circ per\circ D([M, f]) = 0\in KO_n(C^\ast_r(G)),\end{equation} where $[M, f]$ is the element in the $\Spin$-bordism group is the fundamental class in $KO$-theory.
%\end{conj}

%The assumption $n\geq 5$ on the dimension of the manifold is due to the existence of four-dimensional counter-examples to the only if part of the Gromov-Lawson-Lawson Conjecture through the usage of Seiberg-Witten theory. 

One of the fundamental results in obstructive results to positive scalar curvature is due to Rosenberg. He showed that the vanishing of the invariant (\ref{Obstruction Rosenberg}) is a necessary condition for the existence of a Riemannian metric of positive scalar curvature.

\begin{theorem}\label{Theorem Rosenberg} Rosenberg \cite{[Rosenberg1]}. Let $M$ be a closed orientable $n$-manifold with second Stiefel-Whitney class $w_2(M) = 0$. If there is a Riemannian metric $(M, g)$ whose scalar curvature satisfies $\Scal_g > 0$,  then $A\circ per \circ D([M, f]) = 0 \in KO_n(C^\ast_r(G))$. 
\end{theorem}

If the group $G$ satisfies the Strong Novikov Conjecture, i.e., if the assembly map\begin{equation}\label{Assembly Map Injective}A: KO_\ast(BG)\rightarrow KO_\ast(C^\ast_r(G))\end{equation} in (\ref{Sequence}) is injective, then $per \circ D([M, f]) = 0\in KO_n(BG)$ in the presence of a Riemannian metric $(M, g)$ with $\Scal_g > 0$.

Theorem \ref{Theorem Rosenberg} and the injectivity of the map (\ref{Assembly Map Injective}) shift focus onto manifolds whose fundamental classes have nontrivial images in $ko_n(BG)$ and $KO_n(BG)$. An essential $4$-manifold satisfies $f_\ast([M]_K)\neq 0 \in K_n(BG))\otimes \Q$, where $[M]_K\in K_n(M)$ is the $K$-theoretical fundamental class of $M$; see \cite[Definition 2.5]{[Hanke]}. Indeed, every orientable 4-manifold admits a $\Spin^{\C}$-structure and the homological Chern character isomorphism $ch: K_n(M)\otimes \Q \rightarrow H_n(M; \Q)$ in this case yields $ch([M]_K) = [M] + c$ for $c\in H_i(M; \Q)$ with $i\leq n$, and $[M]$ is the fundamental class of $M$ in singular homology; see \cite[Page 284]{[Hanke]}, \cite[Section 4]{[HankeKotschickRoeSchick]}. If the second Stiefel-Whitney class of an essential n-manifold $M$ with fundamental group $G$ is zero, then the fundamental class of $M$ in connective $KO$-homology $[M]_{ko}:= D([M, \id])\in ko_n(BG)$ is nonzero; cf. \cite[Section 2.10]{[Stolz]}, \cite[3.3 Proposition]{[BolotovDranishnikov]}. Similarly,  an essential $n$-manifold $M$ with $w_2(M) = 0$ yields a nontrivial fundamental class in periodic $KO$-homology $per \circ D ([M, f]) =  f_{\ast}([M]_{KO}])\neq 0 \in KO_n(BG)$. We conclude that Theorem \ref{Theorem Rosenberg} has the following implication.

\begin{proposition}\label{Proposition Spin Essential}Let $M$ be an essential $n$-manifold and suppose $w_2(M) = 0$ . If the fundamental group $\pi_1(M)$ satisfies the Strong Novikov Conjecture, then $M$ does not admit a Riemannian metric of positive scalar curvature. 
\end{proposition}

Gromov-Lawson \cite{[GromovLawson1]} introduced the concept of enlargeability to employ the role of the fundamental group to obstruct positive scalar curvature. Recall that a closed orientable Riemannian $n$-manifold $(M, g)$ is enlargeable if for every $\epsilon > 0$, there is a Riemannian cover $\hat{M}\rightarrow M$ and an $\epsilon$-contracting map $\hat{M}\rightarrow S^n(1)$ to the unit $n$-sphere that is constant outside a compact subset and of non-zero degree. The reader is referred to \cite{[GromovLawson2], [GromovLawson3]} for details.

\begin{theorem}\label{Theorem GromovLawson} Gromov-Lawson \cite[Theorem A]{[GromovLawson2]}, \cite[Section 12]{[GromovLawson3]}, Cecchini-Schick \cite[Theorem A]{[CecchiniSchick]}, Schoen-Yau \cite{[SchoenYau]}. Suppose $M$ is an enlargeable $n$-manifold with $n\leq 8$. There is no Riemannian metric $(M, g)$ such that $\Scal_g > 0$.
\end{theorem}

When coupled with other interesting results, Theorem \ref{Theorem GromovLawson} poses constrains on the homotopy types of manifolds of positive scalar curvature in terms of their classifying map. A result of Hanke-Schick implies that $M$ is an essential $n$-manifold whenever $M$ is a closed oriented enlargeable $n$-manifold \cite[Theorem 5.1]{[HankeSchick]}. The converse of such statement does not hold: Brunnbauer-Hanke have constructed examples of essential $n$-manifolds that are not enlargeable for $n\geq 4$ in \cite[Theorem 1.5]{[BrunnbauerHanke]}.

%We mention that work of Brunnbauer-Hanke provides necessary and sufficient conditions for a closed orientable $n$-manifold with free abelian fundamental group $\Z^k$ to be enlargeable \cite[Theorem 4.2]{[BrunnbauerHanke]}. 

%\begin{theorem}\label{Theorem BrunnbauerHanke}Brunnbauer-Hanke \cite[Theorem 4.2]{[BrunnbauerHanke]}. Let $M$ be a closed orientable $n$-manifold with fundamental group $\pi_1(M) = \Z^k$ for $1\leq n \leq k$. The manifold $M$ is enlargeable if and only if it is essential.
%\end{theorem}

The minimal stable hypersurfaces machinery of Schoen-Yau is a powerful obstruction to the existence of a metric of positive scalar curvature \cite[Proof of Theorem 1]{[SchoenYau]}. The setting is as follows. Suppose $M$ is a closed oriented $n$-manifold with $H^1(M; \Z)\neq 0$ and $3\leq n \leq 7$, and suppose $N\subset M$ is a closed stable minimal hypersurface such that its homology class $[N] \neq 0\in H_{n - 1}(M; \Z)$ is dual to a nonzero class in the first cohomology group. If there is a Riemannian metric $(M, g)$ with $\Scal_g > 0$, then there is a Riemannian metric of positive scalar curvature on the hypersurface $N$ within the conformal class of the induced metric.

Schick observed that iterations of Schoen-Yau's results on sequences of stable minimal hypersurfaces provides the following formal obstruction \cite[\S 1]{[Schick]}. We denote by $H_m^+(X; \Z)$ the subset of the mth-homology group $H_m(X; \Z)$ of any space $X$ that consists of classes $f_\ast([M])$ where $f:M\rightarrow X$ for $(M, g)$ that satisfies $\Scal_g >0$.

\begin{theorem}\label{Theorem Schick} Schick \cite[Corollary 1.5]{[Schick]}. Let $X$ be any space and $\theta\neq 0\in H^1(X; \Z)$. The cap product with $\theta$ induces a map\begin{equation}\theta\cap: H_m(X; \Z)\rightarrow H_{m - 1}(X; \Z)\end{equation}that sends\begin{equation}H_m^+(X; \Z)\rightarrow H_{m - 1}^+(X; \Z)\end{equation}for $3\leq m\leq 7$.
\end{theorem}

The well-known necessary conditions for enlargeability \cite[Corollary C]{[GromovLawson1]}, \cite[\S 5]{[GromovLawson3]}, \cite[Theorem 5.3, Theorem 5.4]{[LawsonMichelsohn]} along with Theorem \ref{Theorem GromovLawson} along Theorem \ref{Theorem Schick} yield the following result.

\begin{proposition}\label{Proposition NPSC} Let $M$ be an essential $n$-manifold. Suppose that the fundamental group $\pi_1(M) = G$ either satisfies Property NPSC of Definition \ref{Definition Property NPSC} or that it is isomorphic to the fundamental group of a closed orientable Riemannian $4$-manifold $(N, g)$ of nonpositive sectional curvature. There is no Riemannian metric on $M$ of positive scalar curvature. In particular, the set\begin{equation}\label{Set PSC}\mathcal{G}_{> 0}(G, \alpha) = \mathcal{G}_{> 0}^1(G, \alpha) \sqcup \mathcal{G}_{> 0}^2(G, \alpha)  \sqcup \mathcal{G}_{> 0}^3(G, \alpha)\end{equation} is empty for every $\alpha \neq 0 \in H_4(G)$.
\end{proposition}

\begin{proof} The proof of Proposition \ref{Proposition NPSC} in the case of a group $G$ that satisfies Property NPSC is due to Schick \cite[\S 2]{[Schick]}; see Theorem \ref{Theorem Schick}. Let $f:M\rightarrow BG$ be the reference map into the classifying space $BG$. The universal coefficient theorem implies that the induced map on cohomology $f^\ast: H^1(G)\rightarrow H^1(M; \Z)$ is an isomorphism given that the induced map in homology $f_\ast: H_1(M; \Z)\rightarrow H_1(G)$ is an isomorphism. Hence, $H^1(M; \Z) = \Z^s$ for $s\geq 2$ and we have generators $\{\theta_1, \theta_2\}$ that satisfy $\theta_1\cap \theta_2\cap f_\ast(M])\neq 0 \in H_2(BG; \Z)$. We employ Schoen-Yau's stable minimal hypersurfaces method \cite{[SchoenYau]} and proceed by contradiction. Suppose there is a metric of positive scalar curvature $(M, g)$ and $f_\ast([M]) = \alpha \neq 0\in H_4(BG) = H_4(G)$. This implies that there is a non-trivial element in $H^+_4(BG)$. A pair of applications of Theorem \ref{Theorem Schick} implies the existence of an element $a\neq 0 \in H_2^+(BG)$. However, the only closed oriented 2-manifold that admits a metric of positive scalar curvature is the 2-sphere, and any map $g: S^2\rightarrow BG$ is nullhomotopic since the second homotopy group of the Eilenberg-MacLane space $BG$ is trivial. We then have that $g_\ast([S^2]) = 0\in H_2(BG)$ and $H_2^+(BG) = 0$. Therefore, we conclude that a closed oriented 4-manifold $M$ whose fundamental group has non-trivial first cohomology group and $f_{\ast}([M])\neq 0 \in H_4(G)$ does not admit a Riemannian metric of positive scalar curvature.

We now argue the case where $\pi_1(M) = \pi_1(N)$ for a closed orientable Riemannian 4-manifold $(N, g_N)$ with $\sec_{g_N}\leq 0$. The 4-manifold $N$ is enlargeable by \cite{[GromovLawson1], [GromovLawson3]}; cf. \cite[Theorem 5.4]{[LawsonMichelsohn]}. The Hadamard-Cartan theorem implies that the universal cover $\widetilde{N} = \R^n$ is contractible and $N$ is an aspherical 4-manifold. Therefore, the Eilenberg-MacLane space is $B\pi_1(M) = N$. If $f_\ast([M])\neq 0\in H_4(B\pi_1(M)) = H_4(N)$, there is a non-zero degree map $f:M\rightarrow N$ and $M$ is an enlargeable 4-manifold by \cite{[GromovLawson1], [GromovLawson3]}, \cite[Theorem 5.3]{[LawsonMichelsohn]}. Theorem \ref{Theorem GromovLawson} implies that there is no metric on $M$ with positive scalar curvature and the set (\ref{Set PSC}) is empty for every $\alpha \neq 0\in H_4(N)$. \end{proof}

Theorem \ref{Theorem Schick} is particularly efficient on essential manifolds whose fundamental group is a right-angled Artin group (RAAG).

\begin{example}\label{Example RAAGs}RAAG's with non-zero fourth homology group satisfy Property NPSC of Definition \ref{Definition Property NPSC}. A right-angled Artin group is a group with a presentation
\begin{equation}\label{RAAG Presentation}G = \langle g_1, \ldots, g_s: g_ig_j = g_j g_i \forall i\neq j \rangle.\end{equation}The presentation (\ref{RAAG Presentation}) is specified by a defining graph $\Gamma$, whose vertices are labelled by the generators $\{g_1, \ldots, g_s\}$ and an edge of $\Gamma$ that connects two of them $g_i$ and $g_j$ exists if and only if $g_ig_j = g_jg_i$; see the survey article of Charney \cite{[Charney]} for further details on these groups.
The Eilenberg-MacLane space $BG$ is the Salvetti complex \cite{[Salvetti]} \cite[Section 3.6]{[Charney]}, and Charney-Davis used it to compute its cohomology ring in \cite{[CharneyDavis]}. The chain complex of $BG$ injects into the chain complex of a $k$-torus and given that all chain and cochain maps in the chain complex of a torus are zero, the only nontrivial cup products in the cohomology ring $H^\ast(G)$ arise from the commuting generators in the presentation (\ref{RAAG Presentation}); see \cite{[Charney], [CharneyDavis], [Hildum]}. The vertices of $\Gamma$ represent generators $\{\gamma_1, \ldots, \gamma_s\}$ of $H^1(G) = H^1(BG)$ while its edges represent generators of $H^2(G) = H^2(BG)$ that arise as cup products $\gamma_i\cup \gamma_j$. Hence, $b_1(G)$ is equal to the number of vertices of $\Gamma$ and $b_2(G)$ is equal to the number of edges. In more generality, a subgraph of order $k$, known as a $k$-clique, represents the generators of $H^k(G)$ and these arise as cup products of $k$ generators of the first cohomology group of $G$. The $k$th Betti number $b_k(G)$ equals the number of $k$-cliques. Regarding the fourth cohomology group, a 4-clique of $\Gamma$ represents the generators of $H^4(G)$, which arise as cup products $\gamma_{i_1i_2i_3i_4}:= \gamma_{i_1}\cup\gamma_{i_2}\cup \gamma_{i_3}\cup \gamma_{i_4}$ and a non-zero element $g_{i_1i_2i_3i_4}\neq 0 \in H_4(G)$ satisfies $\langle g_{i_1i_2i_3i_4}, \gamma_{i_1i_2i_3i_4}\rangle = \delta_{l_1, l_2}$; interesting examples of graphs can be found in a paper of Hildum \cite{[Hildum]}. 

Given an essential 4-manifold whose fundamental group $G$ is a right-angled Artin group, we have that $H_4(G)\neq 0$ and there is a 4-clique in the defining graph $\Gamma$. This implies that we can choose elements $\gamma_1, \gamma_2\in H^1(G)$ that correspond to commuting generators of (\ref{RAAG Presentation}) such that cap product $\gamma_1\cap \gamma_2\cap f_\ast([M])$ is equal to a non-zero cup product of elements $\gamma_i, \gamma_j\in H^1(G)$ for $f_\ast([M])\neq 0 \in H_4(G)$; cf. \cite[Example 2.2]{[Schick]}. We conclude that a right-angled Artin group $G$ with $H_4(G)\neq 0$ satisfies the NPSC property of Definition \ref{Definition Property NPSC}. 

The summarize this discussion as a lemma.\begin{lemma} An essential 4-manifold whose fundamental group is a right-angled Artin group does not admit a Riemannian metric of positive scalar curvature.

\end{lemma}

%RAAG's are known to satisfy Hypothesis (\ref{Hypothesis 1}).

\end{example}

Chodosh-Li have shown that a closed aspherical 4-manifold does not admit a Riemannian metric of positive scalar curvature \cite[Theorem 2]{[ChodoshLi]} following a strategy of Schoen-Yau \cite[Theorem 6]{[SchoenYau2]}. Gromov has recently shown that an essential 5-manifold whose fundamental group is isomorphic to the fundamental group of a closed aspherical 5-manifold does not admit a Riemannian metric of positive scalar curvature. We conclude this section with the observation that Gromov's five-dimensional result has the following implication regarding essential 4-manifolds; cf. Chodosh-Li \cite[Theorem 2]{[ChodoshLi]}.

\begin{lemma}\label{Lemma Gromov}Let $M$ be an essential 4-manifold whose fundamental group is isomorphic to the fundamental group of a closed orientable aspherical 4-manifold $M_0$. There is no Riemannian metric on $M$ of positive scalar curvature.

\end{lemma}

\begin{proof}We proceed by contradiction and suppose that there is a Riemannian metric $(M, g)$ with scalar curvature $\Scal_g > 0$. This assumption implies that the product Riemannian 5-manifold $(M\times S^1, g + d\theta^2)$ has positive scalar curvature, and there is a non-zero degree map $f:M\times S^1\rightarrow M_0\times S^1$ to the aspherical 5-manifold $M_0\times S^1$. This is a contradiction by \cite[Main Theorem]{[Gromov]}. We conclude that $M$ does not admit a Riemannian metric of positive scalar curvature. 

\end{proof}

\section{Cut-and-Paste constructions, positive scalar curvature and prescribed fundamental group}\label{Section Constructions}

The production of 4-manifolds that realize the invariant $q_{> 0}(G)$ and that fully determine the set (\ref{Set PSC}) is carried out by employing the following fundamental result.

\begin{theorem}\label{Theorem PSC} Gromov-Lawson \cite[Theorem A]{[GromovLawson1]}, Schoen-Yau \cite[Section 2]{[SchoenYau]}. Let $(M, g)$ be a compact Riemannian manifold with $\Scal_g > 0$. Any manifold that is obtained from $M$ by performing surgeries in codimension at least three also admits a Riemannian metric with positive scalar curvature. 
\end{theorem}

Carr has shown that the fundamental group is not an obstruction for the existence of a Riemannian metric of positive scalar curvature for orientable manifolds of dimension four \cite[Corollary 2]{[Carr]}. His proof consists of constructing a manifold with a Riemannian metric of positive scalar curvature and prescribed finitely presented fundamental group as the boundary of a 2-complex embedded in $\R^5$. A different proof of Carr's result follows from the work of Kervaire \cite{[Kervaire]} and Theorem \ref{Theorem PSC}. We state a strengthened version of Carr's result which include several topological properties of the manifolds involved that were not considered in \cite{[Carr]}. 

\begin{theorem}\label{Theorem Carr} Let $G$ be a group with finite presentation\begin{equation}\label{Group}\left\langle g_1, \ldots, g_s : r_1, \ldots, r_t\right\rangle \end{equation} and suppose $n\geq 4$. There exists a closed smooth stably-parallelizable $n$-manifold $M^n(G)$ such that\begin{itemize}\item the fundamental group of $M^n(G)$ is isomorphic to $G$;
%\item if $G$ is not isomorphic to a free product $G_1\ast G_2$ of nontrivial groups, then $M^n(G)$ is an irreducible $n$-manifold.
\item if $n$ is even, the Euler characteristic is $\chi(M^n(G)) = 2 - 2\Def(G)$ for $\Def(G) = s - t$,
\item there exists a Riemannian metric $(M^n(G), h)$ of positive scalar curvature $\Scal_h > 0$, and
\item the classifying map $f:M^n(G)\rightarrow BG$ satisfies\begin{equation}\label{ClassifyMap}f_{\ast}([M^n(G)]) = 0 \in H_n(BG).\end{equation} 
\end{itemize}
\end{theorem}

Both Carr's and Kervaire's constructions produce a closed stably parallelizable n-manifold equipped with a null bordism along with the classifying map (\ref{ClassifyMap}). Such null bordism is the culprit for the existence of a metric of positive scalar curvature in terms of Theorem \ref{Theorem GromovLawson}. Moreover, although the Euler class of $M^n(G)$ need not be zero, other characteristic numbers are trivial such as the signature of the 4-manifold $M^4(G)$. This 4-manifold realizes the modified Hausmann-Weinberger invariant $q_{> 0}(G)$ in Theorem \ref{Theorem A} and Theorem \ref{Theorem B} under the assumption that the relation $b_1(G) - b_2(G) = \Def(G)$ holds.

We comment on the proof of Theorem \ref{Theorem Carr} in Remark \ref{Remark Carr}, but we point out first that Theorem \ref{Theorem PSC} makes it possible to extend Carr's result to non-orientable manifolds. 

\begin{theorem}\label{Theorem NonOrientable Carr} Let $G$ be a group with finite presentation\begin{equation}\label{Group}\left\langle g_1, \ldots, g_s : r_1, \ldots, r_t\right\rangle \end{equation} that contains a subgroup of index two, and suppose $n\geq 4$.  There exists a closed smooth non-orientable  $n$-manifold $N^n(G)$ such that\begin{itemize}\item the fundamental group of $N^n(G)$ is isomorphic to $G$;
\item if $n$ is even, the Euler characteristic is $\chi(M^n(G)) = 2 - 2\Def(G)$ for $\Def(G) = s - t$,
%\item if $G$ is not isomorphic to a free product $G_1\ast G_2$ of nontrivial groups, then $M^n(G)$ is an irreducible $n$-manifold.
%\item if $n$ is even, the Euler characteristic is $\chi(M^n(G)) = 2 - 2\Def(G) - j$ for $\Def(G) = s - t$,
\item there exists a Riemannian metric $(N^n(G), h)$ of positive scalar curvature $\Scal_h > 0$, and
\item the classifying map $f:M^n(\tilde{G})\rightarrow B\tilde{G}$ of the the orientation 2-cover $M^n(\tilde{G})\rightarrow N^n(G)$ satisfies\begin{equation}\label{ClassifyMap}f_{\ast}([M^n(\tilde{G})]) = 0 \in H_n(B\tilde{G}).\end{equation} 
\end{itemize}
\end{theorem}

%\begin{theorem}\label{Theorem Nonorientable Carr} Let $G$ be a group that contains a subgroup $H\subset G$ of index two and which has a finite presentation\begin{equation}\label{NGroup}G = \left\langle g_1, \ldots, g_s | r_1, \ldots, r_t\right\rangle.\end{equation}Suppose $n\geq 4$. There is a closed smooth non-orientable Riemannian $n$-manifold $(N^n(G), h)$ with scalar curvature $\Scal_h > 0$ and fundamental group $\pi_1(N^n(G))$ isomorphic to $G$. 
%\end{theorem}

The total space of the non-orientable $S^{n - 1}$-bundle over the circle is denoted by $S^{n - 1}\widetilde{\times} S^1$.

\begin{proof} The first step is to construct a closed non-orientable Riemannian $n$-manifold $(X(F_s), g_0)$ with positive scalar curvature $\Scal_{g_0} > 0$ and a free fundamental group $\pi_1(X(F_s))$ on $s$ generators, where $s$ is the number of generators in the presentation (\ref{Group}) of $G$. The fundamental group of the connected sum\begin{equation}\label{Connected Sum Step 1}X(F_s):= S^{n - 1}\widetilde{\times} S^1 \# S^1\times S^{n - 1}\# \cdots \# S^1\times S^{n - 1}\end{equation} of the non-orientable $(n - 1)$-sphere bundle with $s - 1$ copies of $S^1\times S^{n - 1}$ is the free group $\pi_1(X(F_s)) = F_s = \langle g_1, \ldots, g_s\rangle$ on $s$ generators. The claim that the fundamental group of $X_s$ is the free group on $s$ generators follows from the Seifert-van Kampen theorem under the assumption that $n\geq 3$. The hypothesis on the existence of a subgroup of index two of $G$ implies that $s\in \N$. The fixed-point free orientation-reversing involution $S^1\times S^{n - 1}\rightarrow S^1\times S^{n - 1}$ that is given by $(\exp 2\pi i \theta, x)\mapsto (-\exp 2 \pi i \theta, -x)$ is an isometry of the Riemannian product of canonical metrics $(S^1\times S^{n - 1}, d\theta^2 + g_{S^{n - 1}})$, and the scalar curvature of the latter is positive. We equip $S^{n - 1}\widetilde{\times} S^1$ with the quotient metric, which has positive scalar curvature. Theorem \ref{Theorem PSC} implies that there is a Riemannian metric of positive scalar curvature on the connected sum (\ref{Connected Sum Step 1}) for every $s$. 
 
For each relation $r_j$ in the group presentation (\ref{Group}), take an embedded loop $\gamma_j\subset X(F_s)$ that represents the homotopy class that corresponds to $r_j$. This yields a set of embedded loops\begin{equation}\label{Set Loops}\{\gamma_1, \ldots, \gamma_r\}\end{equation} along a one-to-one assigment with the set $\{r_1, \ldots, r_t\}$ for relations in  (\ref{Group}).  Since an embedded loop has codimension at least two in the $n$-manifold (\ref{Connected Sum Step 1}), transversality guarantees that the set (\ref{Set Loops}) consists of pairwise disjoint embedded loops $\gamma_j \subset X(F_s)$ for $j = 1, \ldots, t$. The possible diffeomorphism type of the tubular neighborhoods of embedded loops in an $n$-manifold $M$ are collected in the following lemma; see  \cite[Chapter 5]{[DavisKirk]} for a definition of the orientation character homomorphism $\rho_M: \pi_1(M)\rightarrow \Z/2$.
 
 \begin{lemma}\label{Lemma Tubular NbhD} Let $M$ be a smooth $n$-manifold and $\gamma\subset M$ be an embedded loop. 
 
If the orientation character of $\gamma$ is zero, the tubular neighborhood $\nu(\gamma)\subset M$ is diffeomorphic to $S^1\times D^{n - 1}$ .
 
If the orientation character of $\gamma$ is non-zero, the tubular neighborhood $\nu(\gamma)\subset M$ is diffeomorphic to $D^{n - 1}\widetilde{\times} S^1$.
 \end{lemma} 
 
The second step is as follows. We carve out all disjoint tubular neighborhoods of the loops (\ref{Set Loops}) to obtain a codimension zero compact submanifold $X_s\subset X(F_s)$ with $t$ boundary components. Since it was assumed that $G$ contains a subgroup of index two, we can choose the loops (\ref{Set Loops}) to have trivial orientation character. The boundary of the compact $n$-manifold $X_s$ is given by the disjoint union of $t$ copies of\begin{equation}\partial(\nu(\gamma_ i)) = \partial (S^1\times D^{n - 1}) = S^1\times S^{n - 2}\end{equation} by Lemma \ref{Lemma Tubular NbhD}. The inclusion\begin{equation}\label{Inclusion Map}X_s:= X(F_s)\backslash \overset{t}{\underset{j = 1}{\bigsqcup}} \nu(\gamma_j)\hookrightarrow X(F_s)\end{equation} induces a group homomorphism\begin{equation}\label{Group Epi}\pi_1(X_s)\rightarrow \pi_1(X(F_s))\end{equation}that is surjective since the tubular neighborhoods are disjoint. We now argue that the group epimorphism (\ref{Group Epi}) is also a monomorphism provided $n\geq 4$. Indeed, suppose a given loop $\gamma_{i_0}\subset X_s$ bounds a 2-disk that is contained in $X(F_s)$. By a transversality argument, we can push such 2-disk away from all the other loops in (\ref{Set Loops}) since $n > 2 + 1$. Thus, the loop $\gamma_{i_0}$ bounds a 2-disk in the manifold $X_s$ and it is a trivial element in the fundamental group $\pi_1(X_s)$. Thus, we conclude that (\ref{Group Epi}) is a group isomorphism. 

The third step is to cap off each boundary component of $X_s$ with a copy of $D^2\times S^{n - 2}$ in order to construct the closed smooth non-orientable $n$-manifold\begin{equation}\label{n1manifold}N^n(G):= \Big(X(F_s)\backslash \overset{t}{\underset{j = 1}{\bigsqcup}} \nu(\gamma_j)\Big)\bigcup \Big(\overset{t}{\underset{j = 1}{\bigsqcup}} (D_j^2\times S^{n - 2})\Big).\end{equation}The Seifert-van Kampen theorem is used to conclude that the manifold $N^n(G)$ has fundamental group isomorphic to (\ref{Group}). If $n$ is even, we use a standard homological computation using a Mayer-Vietoris sequence to conclude that the Euler characteristic is $\chi(N^n(G)) = 2 - 2\Def(G)$ for the deficiency $\Def(G) = s - t$ of the presentation (\ref{Group}) of the group $G$.

The fourth and last step is to appeal to Theorem \ref{Theorem PSC} to conclude that there is a Riemannian metric $(N^n(G), h)$ with scalar curvature $\Scal_h > 0$. The Riemannian product  $(D^2\times S^{n - 2}, g_{D^2} + g_{S^{n - 2}})$ of a round hemisphere with a round $(n - 2)$-sphere has positive scalar curvature and the gluing construction (\ref{n1manifold}) of $N^n(G)$ is along submanifolds of codimension $n - 1\geq 3$. Theorem \ref{Theorem PSC} guarantees the existence of a Riemannian metric $(N^n(G), h)$ of positive scalar curvature $\Scal_h > 0$.

%Consider the involution $\sigma: S^2\times S^{n - 2}\rightarrow S^2\times S^{n - 2}$ given by $(x, y)\mapsto (r(x), - y)$ where $r:S^2\times S^2$ is a reflection with one -1 eigenvalue. The orbit space $S^2\times S^{n - 2}/\sigma = S^2\widetilde{\times} \R P^{n - 2}$  of this fixed-point free involution is the total space of a non-orientable 2-sphere bundle over the real projective (n - 2)-space. This non-orientable bundle decomposes as the double\begin{equation}\label{Bundle F}S^2\widetilde{\times} \R P^{n - 2} = (D^2\widetilde{\times} \R P^{n - 2}) \cup (D^2\widetilde{\times} \R P^{n - 2})\end{equation} of a non-trivial non-orientable $D^2$-bundle over $\R P^{n - 2}$ whose boundary is diffeomorphic to $\partial (D^2\widetilde{\times} \R P^{n - 2}) = S^{n - 1}\widetilde{\times} S^1$.

\end{proof}

\begin{remark}\label{Remark Carr}A modification to the first step in the proof of Theorem \ref{Theorem NonOrientable Carr} yields a proof of Theorem \ref{Theorem Carr}. We substitute the closed 4-manifold (\ref{Connected Sum Step 1}) with the closed orientable 4-manifold\begin{equation}X(F_s) = S^1\times S^{n - 1}\#\cdots \#S^1\times S^{n - 1}\end{equation}given by the connected sum of $s\in \Z_{\geq 0}$ copies of $S^1\times S^{n - 1}$. The remaining steps in the proof of Theorem \ref{Theorem NonOrientable Carr} are performed verbatim, and a proof of Theorem \ref{Theorem Carr} is obtained; cf. \cite{[Kervaire]}. 

%The Euler characteristic of the closed smooth non-orientable $n$-manifold $N^n(G)$ that was constructed in the proof of Theorem \ref{Theorem Nonorientable Carr} is\begin{equation}\chi(N^n(G)) = 2 - 2s + 2 t_1 + t_2\end{equation} for $n$ even.

\end{remark}

\section{Proofs}

\subsection{Proof of Theorem \ref{Theorem A}}We show that $\mathcal{G}^2_{> 0}(G) = \mathcal{G}^2_{> 0}(G, 0)$, and then proceed to compute the latter set as it was done in Example \ref{Geography Surface Groups}. As it was discussed in Section \ref{Section NPositiveScalarCurvature}, Proposition \ref{Proposition Spin Essential} states that the sets $\mathcal{G}^2_{> 0}(G, \alpha)$ are empty for $\alpha \neq 0$ appealing to work of Rosenberg \cite{[Rosenberg1], [Rosenberg2]}. Let\begin{equation}\label{Original Invariant}q_{\alpha}(G, \sigma):= \Min \{\chi(M): (M, f)\in \mathcal{M}(G, \alpha) \; \text{and} \space \; \sigma(M) = \sigma\}\end{equation}where $\mathcal{M}(G, \alpha)$ is the set of closed smooth oriented 4-manifolds with fundamental group $G$ and $f_\ast([M]) = \alpha\in H_4(G)$. We now recall a well-known proposition from the literature.

\begin{proposition}\label{Proposition Geography Zero}Kirk-Livingston \cite[Proposition 4.1]{[KirkLivingston2]}, Kotschick \cite{[Kotschick]}. 

If $\alpha = 0\in H_4(G)$, then the inequalities\begin{equation}2 - 2(b_1(G) - b_2(G)) + |\sigma|\leq q_{0}(G, \sigma)\leq 2 - 2\Def(G) + |\sigma|\end{equation}hold.
\end{proposition}

Proposition \ref{Proposition Geography Zero} implies that the 4-manifold $M^4(G)$ of Theorem \ref{Theorem Carr} realizes the function (\ref{Relation}) by Hypothesis (\ref{Hypothesis 1}). Theorem \ref{Theorem Lichnerowicz} implies that $\sigma(M) = 0$ for every $M\in \mathcal{M}^2_{> 0}(G)$, and we can conclude that $q_{> 0}^2(G) = 2 - 2\Def(G) = q^2_{> 0, 0}(G, 0)$. We now need to determine the values of pairs $\{(\chi(M), 0)\}$ to compute the geography and conclude the proof of the theorem. Consider the product of two round 2-spheres $(S^2\times S^2, g_{S^2} + g_{S^2})$, whose scalar curvature is positive, Theorem \ref{Theorem PSC} implies that the connected sum $M_k:= M^4(G)\# k(S^2\times S^2)$ admits a Riemannian metric of positive scalar curvature for any nonnegative integer number $k$. A Mayer-Vietoris sequence argument shows that Euler characteristic is $\chi(M_k) = 2k - 2\Def(G)$ and we conclude that $\mathcal{G}^2_{> 0}(G) = \{(2 - 2\Def(G) + 2k, 0) : k\in \Z_{\geq 0}\}$.

\hfill $\square$

\subsection{Proof of Theorem \ref{Theorem B}} The proof is similar of Theorem \ref{Theorem A}. Proposition \ref{Proposition NPSC} implies that the sets $\mathcal{G}_{> 0}(G, \alpha)$ are empty for $\alpha \neq 0$ in the case $H_4(G)\neq 0$ and $G$ satisfies Property NPSC or if $G$ is the fundamental group of a closed orientable nonpositively curved Riemannian 4-manifold. If $G$ is the fundamental group of a closed orientable aspherical 4-manifold, the same conclusion holds by Lemma \ref{Lemma Gromov}. Proposition \ref{Proposition Geography Zero} guarantees that the 4-manifold produced in Theorem \ref{Theorem Carr} realizes the modified Hausmann-Weinberger invariant\begin{equation}q_{> 0}(G) = 2 - 2\Def(G) = q_{> 0, 0}(G, 0) = q_{> 0, 0}^0(G, 0).\end{equation}Using Theorem \ref{Theorem PSC} and taking connected sums with copies of $\mathbb{CP}^2$ and/or $\overline{\mathbb{CP}^2}$ we conclude that the equality\begin{equation}q_{> 0}(G, \sigma) = |\sigma| + 2 - 2\Def(G) = q_{> 0, 0}(G, \sigma)\end{equation}holds and $\mathcal{G}_{> 0}(G) = \{(2 - 2\Def(G) + |k|, k) : k\in \Z_{\geq 0}\}$. Analogously, we conclude that $\mathcal{G}_{> 0}^2(G) = \{(2 - 2\Def(G) + 2k, 0) : k\in \Z_{\geq 0}\}$ by invoking Theorem \ref{Theorem Lichnerowicz} and taking connected sums with copies of $S^2\times S^2$ equipped with a metric of positive scalar curvature. 

\hfill $\square$

\subsection{Proof of Proposition \ref{Proposition C}} Notice that the 4-manifold with fundamental group $G_{p, q}$ that is produced by Theorem \ref{Theorem Carr} has Euler characteristic equal to 4. A closed smooth 4-manifold with a Riemannian metric of postive scalar curvature with fundamental group isomorphic to (\ref{Examples Groups}), Euler characteristic equal to 2, signature equal to zero and trivial second Stiefel-Whitney class is constructed as follows. Consider the loop\begin{equation}\label{Loop p}\gamma_p:= S^1\times \{z\}\subset S^1\times L(q, 1)\end{equation}whose homotopy class corresponds to $p$ times the generator of $\pi_1(S^1\times \{z\}) = \Z$; $\{z\}$ is a point in the Lens space $L(q, 1)$ with cyclic fundamental group $\Z/q$. We carve out a tubular neighborhood of (\ref{Loop p}) and cap off the boundary with a copy of $D^2\times S^2$ to obtain a closed smooth orientable 4-manifold\begin{equation}\label{Manifold pq}M_{p, q}:= (S^1\times L(q, 1)\setminus \nu(\gamma_p)) \cup (D^2\times S^2).\end{equation} The Seifert-van Kampen theorem implies that the fundamental group of $M_{p, q}$ is isomorphic to $G_{p, q}$ and Theorem \ref{Theorem GromovLawson} implies that the 4-manifold (\ref{Manifold pq}) admits a metric of positive scalar curvature since the canonical metric $(S^1\times L(p, 1), d\theta^2 + g_{L(p, 1)})$ has nonnegative sectional curvature and positive scalar curvature. Proposition \ref{Proposition Geography Zero} implies $q_{> 0}(G_{p, q}) = 2$, and $M_{p, q}$ realizes this value; cf. \cite[Theorem 2]{[KirkLivingston1]}. Taking connected sums with $\mathbb{CP}^2$ and/or $\overline{\mathbb{CP}^2}$ we conclude $q_{> 0, 0}(G_{p, q}, \sigma) = 2 + |\sigma|$ and $\mathcal{G}_{> 0}(G_{p, q}) = \{(2 + |k|, k): k\in \Z\}$. Analogously, taking connected sums with copies of $S^2\times S^2$, we conclude that $\mathcal{G}_{> 0}^2(G_{p, q}) = \{(2k + 2, 0): k\in \Z_{\geq 0}\}$.

\hfill $\square$

\bibliographystyle{abbrv}
\bibliography{geography} 

\end{document}